\documentclass[a4paper,11pt,twoside,notitlepage]{amsart}

\usepackage{amsmath,amssymb,amsthm,mathtools,mathrsfs}
\usepackage[shortlabels]{enumitem}
\usepackage{microtype}
\usepackage{xcolor}

\definecolor{mutedgray}{HTML}{667085}
\usepackage[ocgcolorlinks,linkcolor=blue,citecolor=blue,urlcolor=blue]{hyperref}
\usepackage{url}
\numberwithin{equation}{section}
\mathtoolsset{showonlyrefs}
\allowdisplaybreaks

\newtheorem{theorem}{Theorem}[section]
\newtheorem{proposition}[theorem]{Proposition}
\newtheorem{lemma}[theorem]{Lemma}
\newtheorem{corollary}[theorem]{Corollary}
\theoremstyle{remark}

\newcommand{\p}{\partial}

\newcommand{\para}[1]{\vspace{3mm}\noindent\textbf{#1.}}

\theoremstyle{definition}
\newtheorem{definition}[theorem]{Definition}
\newcommand{\R}{\mathbb R}
\newcommand{\N}{\mathbb N}
\newcommand{\wt}{\widetilde}
\newcommand{\eps}{\varepsilon}
\DeclareMathOperator{\supp}{supp}

\DeclareMathOperator{\Tr}{Tr}
\title[Fractional Calder\'on's problem at a single point]{The Calder\'on problem for the fractional Schr\"odinger equation at a single point}

\author[B. Harrach]{Bastian Harrach}
\address{Institute for Mathematics, Goethe University Frankfurt, Germany}
\email{harrach@math.uni-frankfurt.de}

\author[Y.-H. Lin]{Yi-Hsuan Lin}
\address{Department of Applied Mathematics, National Yang Ming Chiao Tung University, Hsinchu, Taiwan \& Fakult\"at f\"ur Mathematik, University of Duisburg-Essen, Essen, Germany}
\email{yihsuanlin3@gmail.com}

\hypersetup{
	pdftitle={The Calder\'on problem for the fractional Schrödinger equation at a single point},
	pdfauthor={Bastian Harrach, Yi-Hsuan Lin},
	pdfsubject={Determination of a potential from exterior observations at a fixed point},
	pdfkeywords={Fractional Laplacian, Calder\'on problem, exterior measure data, unique continuation, monotonicity}
}

\begin{document}

	\begin{abstract}
		We revisit the Calder\'on problem for the fractional Schr\"odinger equation and prove that the potential is determined by Dirichlet-to-Neumann maps at a single exterior point, with exterior inputs varying in a disjoint open set. This holds for continuous potentials for all $0<s<1$ and for bounded potentials when $s\geq1/4$. For ordered nonnegative bounded potentials, one fixed nonzero nonnegative input and one point measurement suffice for every $0<s<1$.
	\end{abstract}
	
	\maketitle

	\section{Introduction}\label{sec:introduction}
	Let $\Omega\subset\R^n$, $n\geq1$, be a bounded Lipschitz domain, let $0<s<1$, and set $\Omega_e=\R^n\setminus\overline\Omega$. For a real-valued $q\in L^\infty(\Omega)$, consider the exterior problem
	\begin{equation}\label{eq:main}
		\begin{cases}
			((-\Delta)^s+q)u=0 & \text{in }\Omega,\\
			u=f & \text{in }\Omega_e.
		\end{cases}
	\end{equation}
	We assume that zero is not a Dirichlet eigenvalue of $(-\Delta)^s+q$, namely,
	\begin{equation}\label{eq:nonresonance}
		v\in\wt H^s(\Omega),\quad ((-\Delta)^s+q)v=0\text{ in }\Omega
		\implies v=0,
	\end{equation}
	where $\wt H^s(\Omega)$ is the closure of $C_c^\infty(\Omega)$ in $H^s(\R^n)$. Under \eqref{eq:nonresonance}, every $f\in H^s(\R^n)$ gives a unique solution $u_{q,f}\in H^s(\R^n)$ with $u_{q,f}-f\in\wt H^s(\Omega)$; see \cite[Lemma~2.3]{GSU20}. The exterior Dirichlet-to-Neumann (DN) map is $\Lambda_qf=(-\Delta)^su_{q,f}|_{\Omega_e}$. We ask whether the values $(\Lambda_qf)(x_0)$ at one fixed point $x_0\in\Omega_e$, for smooth $f$ supported away from $x_0$, determine $q$.
	
	The fractional Calder\'on problem was introduced in \cite{GSU20}, where uniqueness from Dirichlet and Neumann data on arbitrary exterior open sets was proved using unique continuation and Runge approximation. In \cite{GRSU20}, one fixed nonzero exterior input was shown to be enough when its Neumann data are known on an exterior open set. That result applies to continuous potentials for all $0<s<1$, and to bounded potentials for $s\geq1/4$; the latter case uses unique continuation from sets of positive measure. For stability and instability, see \cite{RS20,ruland2018exponential}, and for monotonicity methods, see \cite{harrach2017nonlocal-monotonicity,harrach2020monotonicity,lin2020monotonicity}.
	
	Other nonlocal inverse problems concern potentials and obstacles \cite{CLL2017simultaneously}, equations with drift \cite{cekic2020calderon}, variable coefficients \cite{GLX, GU2021calder}, higher-order fractional equations \cite{CMRU20}, and semilinear equations \cite{LL2020inverse}. We also refer to \cite{LLR2019calder, LLU2022para} for parabolic equations, \cite{FGKU_2025fractional, FKU24} for problems on closed manifolds, \cite{LZ2023unique} for a porous medium equation, \cite{FL24, LLY25_entangle} for entanglement principles, and \cite{HLW25} for the logarithmic Schr\"odinger equation. Further references are given in \cite{LL25_Integro}.
	
	\para{Main results} Let $W_1,W_2\Subset\Omega_e$ be nonempty open sets with $\overline{W_1}\cap\overline{W_2}=\emptyset$. By Lemma~\ref{lem:exterior-observation}, $\Lambda_qf$ is smooth on $W_2$ for $f\in C_c^\infty(W_1)$. Given a finite signed Borel measure $\mu$ with $\supp\mu\Subset W_2$, we define
	\begin{equation}\label{eq:measurement}
		\mathcal M_{q,\mu}(f)=\int_{W_2}\Lambda_qf\,d\mu.
	\end{equation}
	We fix $\mu$ and let $f$ vary in $C_c^\infty(W_1)$. Taking $\mu=\delta_{x_0}$ gives $\mathcal M_{q,\mu}(f)=(\Lambda_qf)(x_0)$.
	
	\begin{theorem}\label{thm:uniqueness}
		Let $W_1,W_2\Subset\Omega_e$ be nonempty open sets with disjoint closures, and let $\mu\neq0$ be a finite signed Borel measure with $\supp\mu\Subset W_2$. Suppose that $q_1,q_2$ satisfy \eqref{eq:nonresonance} and that either
		\begin{enumerate}[\rm(i)]
			\item\label{item_1_thm} $q_1,q_2\in C^0(\overline\Omega)$ and $0<s<1$, or
			\item\label{item_2_thm} $q_1,q_2\in L^\infty(\Omega)$ and $1/4\leq s<1$.
		\end{enumerate}
		If
		\begin{equation}\label{eq:equal-data}
			\mathcal M_{q_1,\mu}(f)=\mathcal M_{q_2,\mu}(f)
			\quad\text{for every }f\in C_c^\infty(W_1),
		\end{equation}
		then $q_1=q_2$ almost everywhere in $\Omega$. In case \ref{item_1_thm}, equality holds everywhere in $\overline\Omega$.
	\end{theorem}
	
	Given $x_0\in\R^n\setminus\overline{\Omega\cup W_1}$, choose $W_2$ containing $x_0$ and take $\mu=\delta_{x_0}$ in Theorem~\ref{thm:uniqueness}. Thus, $(\Lambda_qf)(x_0)$ for all $f\in C_c^\infty(W_1)$ determines $q$. The measure must be nonzero, but it may have either sign.
	
	If two nonnegative potentials are ordered, a single nonnegative input and one point observation suffice, for every $0<s<1$.
	
	\begin{theorem}\label{thm:single-input-point}
		Let $0<s<1$, let $W\Subset\Omega_e$ be a nonempty open set, and let $x_0\in\Omega_e\setminus\overline W$. Suppose that $q_1,q_2\in L^\infty(\Omega)$ satisfy $0\leq q_2\leq q_1$ almost everywhere in $\Omega$. Fix a nonzero function $f\in C_c^\infty(W)$ with $f\geq0$. If $(\Lambda_{q_1}f)(x_0)=(\Lambda_{q_2}f)(x_0)$, then $q_1=q_2$ almost everywhere in $\Omega$.
	\end{theorem}
	
	Nonnegative potentials satisfy \eqref{eq:nonresonance}. We prove Theorem~\ref{thm:single-input-point} using a comparison identity and positivity of the interior solutions for $f$ and $\delta_{x_0}$. This is related to the monotonicity relations in \cite{harrach2017nonlocal-monotonicity}; here the identity is for a point observation. In particular, if $q_0\geq0$ is known and $q\geq0$ satisfies either $q\geq q_0$ or $q\leq q_0$, equality of the measured values implies $q=q_0$.
	
	For potentials in a finite-dimensional space, one fixed nonnegative input also gives uniqueness from finitely many point observations. The points can be chosen independently of the potentials.
	
	\begin{corollary}\label{cor:finitely_many_points}
		Let $W_1,W_2\Subset\Omega_e$ be nonempty open sets with disjoint closures, and assume that $W_2$ is connected. Let $V\subset L^\infty(\Omega)$ be a real linear space of dimension $d\geq1$, and fix $0\not\equiv f\in C_c^\infty(W_1)$ with $f\geq0$. For almost every $(x_1,\ldots,x_{2d+1})\in W_2^{2d+1}$, the following holds for all $q_1,q_2\in V$ with $q_1,q_2\geq0$ almost everywhere in $\Omega$: if $(\Lambda_{q_1}f)(x_j)=(\Lambda_{q_2}f)(x_j)$ for $j=1,\ldots,2d+1$, then $q_1=q_2$ almost everywhere in $\Omega$.
	\end{corollary}
	
	To prove Theorem~\ref{thm:uniqueness}, we first associate with $\mu$ a smooth source in $\Omega$, using the separation of $\supp\mu$ from $\overline\Omega$. We solve for $v_{q,\mu}\in\wt H^s(\Omega)$ and regard $u_{q,\mu}=v_{q,\mu}+\mu$ as a distribution on $\R^n$. If $\mu$ has an atom, this distribution need not belong to $H^s(\R^n)$.
	
	We relate the Green and Poisson operators by a transposition formula. For $q=0$ and $L^2$ exterior data, a very weak formulation and the corresponding adjoint relation appear in \cite[Definition~3.3 and Lemma~4.2]{AKW19}. We prove the formula for bounded potentials and finite measures. Symmetry yields $\mathcal M_{q,\mu}(f)=\int_{W_1}f\,(-\Delta)^su_{q,\mu}\,dx$; see Proposition~\ref{prop:reciprocity}. Equality of the data and the unique continuation show that the two interior solutions agree. Their equations then give $(q_1-q_2)v=0$ in $\Omega$.
	
	For continuous potentials, the unique continuation completes the argument. For bounded potentials, we prove that $v_{q,\mu}$ cannot vanish on a set of positive measure if $\mu\neq0$. The extension of $u_{q,\mu}$ has finite weighted energy locally above $\Omega$ and satisfies a Robin condition there. This lets us apply the local estimates of \cite[Section~5]{GRSU20}, followed by a rescaling argument.
	
	\para{Organization of the paper} Section~\ref{sec:forward} treats the source problem induced by an exterior measure and proves the duality, reciprocity, positivity, and comparison results. Section~\ref{sec:ucp} proves the unique continuation statements. The proofs of Theorems~\ref{thm:uniqueness} and~\ref{thm:single-input-point} and Corollary~\ref{cor:finitely_many_points} are given in Section~\ref{sec:proof}.
	
	\section{The forward problem}\label{sec:forward}
	
	\subsection{Function spaces and the energy form}
	For a smooth rapidly decreasing function, the fractional Laplacian is given by
	\begin{equation}\label{eq:frac-laplacian}
		(-\Delta)^su(x)=C_{n,s}\,\mathrm{P.V.}\int_{\R^n}
		\frac{u(x)-u(z)}{|x-z|^{n+2s}}\,dz,
	\end{equation}
	where $C_{n,s}=4^s\Gamma(n/2+s)/(\pi^{n/2}|\Gamma(-s)|)$. We use the Fourier definition of $H^t(\R^n)$ for $t\in\R$. If $D$ is open, $H^t(D)=\{v|_D:v\in H^t(\R^n)\}$ has the quotient norm, and $\wt H^t(D)$ is the closure of $C_c^\infty(D)$ in $H^t(\R^n)$. We write $H^{-s}(D)=(\wt H^s(D))^*$. For $0<s<1$, the usual $H^s(\R^n)$ norm is equivalent to the norm given by the $L^2$ norm and the Gagliardo seminorm; see \cite{DNPV12,GSU20}.
	
	Put $X=\wt H^s(\Omega)$, regarding its elements as functions on $\R^n$ that vanish in $\Omega_e$. We set
	\begin{equation}\label{eq:bilinear}
		B_q(v,w)=\int_{\R^n}(-\Delta)^{s/2}v\,(-\Delta)^{s/2}w\,dx
		+\int_\Omega qvw\,dx, \quad \text{for }v,w\in H^s(\R^n).
	\end{equation}
	The form $B_q$ is continuous and symmetric, and the fractional Poincar\'e inequality makes $B_0$ coercive on $X$. Since $X$ is compactly embedded in $L^2(\Omega)$, multiplication by $q$ gives a compact perturbation from $X$ to $X^*$. The Fredholm alternative and \eqref{eq:nonresonance} imply that, for every $h\in X^*$, there is a unique $z\in X$ with $B_q(z,\phi)=\langle h,\phi\rangle$ for all $\phi\in X$, and $\|z\|_{H^s(\R^n)}\leq C\|h\|_{X^*}$; see also \cite[Lemma~2.3]{GSU20}.
	
	We denote this solution by $\mathcal G_qh$. Thus, $\mathcal G_q:X^*\to X$ is the Green operator with zero exterior value. Moreover, symmetry gives
	\begin{equation}\label{eq:sym_of_Green_kernel}
		\int_\Omega(\mathcal G_qh)k\,dx=\int_\Omega h(\mathcal G_qk)\,dx, \quad \text{for }h,k\in L^2(\Omega).
	\end{equation} 
	We use the zero extension when writing $\mathcal G_qh$ outside $\Omega$.
	
	\subsection{The source induced by an exterior measure}\label{sec:measure-forward}
	Let $\mu$ be a real finite signed Borel measure with compact support $K\Subset\Omega_e$, and write $\|\mu\|_{\mathcal M}=|\mu|(\R^n)$. We regard $\mu$ as a distribution on $\R^n$. For $x\in\Omega$, define
	\begin{equation}\label{eq:forcing}
		F_\mu(x)=C_{n,s}\int_K|x-z|^{-n-2s}\,d\mu(z).
	\end{equation}
	The distance between $K$ and $\overline\Omega$ is positive. Differentiating under the integral shows that $F_\mu$ is smooth near $\overline\Omega$ and $\sup_{\overline\Omega}|\partial^\alpha F_\mu|\leq C_\alpha\|\mu\|_{\mathcal M}$ for every multi-index $\alpha$. In particular, $F_\mu\in L^2(\Omega)$.
	
	For $\phi\in C_c^\infty(\Omega)$, formula \eqref{eq:frac-laplacian} and Fubini's theorem give $\int_K(-\Delta)^s\phi\,d\mu=-\int_\Omega F_\mu\phi\,dx$. Hence, $(-\Delta)^s\mu=-F_\mu$ in $\Omega$.
	
	\begin{definition}\label{def:measure-solution}
		An interior solution associated with $\mu$ is a function $v\in X$ satisfying
		\begin{equation}\label{eq:source}
			B_q(v,\phi)=\int_\Omega F_\mu\phi\,dx
			\quad\text{for every }\phi\in X.
		\end{equation}
		When \eqref{eq:nonresonance} holds, we denote this unique solution by $v_{q,\mu}$.
	\end{definition}
	
	The equation for $v_{q,\mu}$ is 
	\begin{equation}
		\begin{cases}
			((-\Delta)^s+q)v_{q,\mu}=F_\mu &\text{ in }\Omega, \\
			v_{q,\mu}=0 &\text{ in }\Omega_e.
		\end{cases}
	\end{equation} 
	We also set $u_{q,\mu}=v_{q,\mu}+\mu$ as a distribution on $\R^n$, identifying $v_{q,\mu}$ with $v_{q,\mu}\,dx$ in this sum. For $\phi\in C_c^\infty(\R^n)$,
	\begin{equation}\label{eq:distribution}
		\langle(-\Delta)^su_{q,\mu},\phi\rangle
		=\int_\Omega v_{q,\mu}(x)(-\Delta)^s\phi(x)\,dx
		+\int_K(-\Delta)^s\phi(z)\,d\mu(z).
	\end{equation}
	Both integrals are finite because $v_{q,\mu}\in L^1(\Omega)$ and $(-\Delta)^s\phi$ is bounded. By \eqref{eq:source} and $(-\Delta)^s\mu=-F_\mu$ in $\Omega$, we have $\langle(-\Delta)^su_{q,\mu},\phi\rangle+\int_\Omega qv_{q,\mu}\phi\,dx=0$ for $\phi\in C_c^\infty(\Omega)$. In this sense $((-\Delta)^s+q)u_{q,\mu}=0$ in $\Omega$, where $qu_{q,\mu}$ means $qv_{q,\mu}$ there. For compactly supported $f\in L^1(\Omega_e)$, we may take $\mu=f\,dx$.
	
	\begin{lemma}[Well-posedness]\label{lem:measure-wellposed}
		Assume \eqref{eq:nonresonance}. For every finite signed Borel measure $\mu$ supported in $K\Subset\Omega_e$, there is a unique interior solution associated with $\mu$. It is given by $v_{q,\mu}=\mathcal G_qF_\mu$ and satisfies $\|v_{q,\mu}\|_{H^s(\R^n)}\leq C\|\mu\|_{\mathcal M}$, where $C$ depends on $\Omega,K,n,s,q$. If $\mu=f\,dx$ with $f\in C_c^\infty(\Omega_e)$, then $u_{q,\mu}=v_{q,\mu}+f$ agrees with the usual energy solution with exterior value $f$.
	\end{lemma}
	
	\begin{proof}
		The bound on $F_\mu$ gives $\|F_\mu\|_{X^*}\leq C\|\mu\|_{\mathcal M}$. Applying $\mathcal G_q$ proves existence, uniqueness, and the stated estimate. For $\mu=f\,dx$ as in the statement, $v_{q,\mu}+f\in H^s(\R^n)$ and $B_q(v_{q,\mu}+f,\phi)=0$ for every $\phi\in X$, since $(-\Delta)^sf=-F_\mu$ in $\Omega$. Uniqueness of the energy solution gives the last assertion.
	\end{proof}
	
	For $z\in\Omega_e$, taking $\mu=\delta_z$ gives $v_{q,\delta_z}=\mathcal G_q(C_{n,s}|\,\cdot-z|^{-n-2s})$. The source is smooth near $\overline\Omega$, while $u_{q,\delta_z}=v_{q,\delta_z}+\delta_z$ need not belong to $H^s(\R^n)$. The estimate of Lemma~\ref{lem:measure-wellposed} applies to its interior part.
	
	\begin{lemma}[Approximation]\label{lem:approximation}
		Suppose that $\mu_k\rightharpoonup^*\mu$ as finite measures and that all these measures are supported in a fixed compact set $K\Subset\Omega_e$. Then $v_{q,\mu_k}\to v_{q,\mu}$ in $H^s(\R^n)$ and $u_{q,\mu_k}\to u_{q,\mu}$ in $\mathscr D'(\R^n)$.
	\end{lemma}
	\begin{proof}
		Weak-star convergence gives $\sup_k\|\mu_k\|_{\mathcal M}<\infty$. For every $x\in\overline\Omega$, the kernel in \eqref{eq:forcing} is continuous on $K$, so $F_{\mu_k}(x)\to F_\mu(x)$. The derivative bounds are uniform, and equicontinuity on the compact set $\overline\Omega$ gives uniform convergence. Thus, $F_{\mu_k}\to F_\mu$ in $X^*$, and continuity of $\mathcal G_q$ gives convergence in $H^s(\R^n)$. Adding $\mu_k\rightharpoonup^*\mu$ proves the distributional convergence.
	\end{proof}
	
	In particular, let $K\Subset W\Subset\Omega_e$ and set $f_k=\rho_{\eps_k}*\mu$, where $\rho$ is a standard nonnegative mollifier and $\eps_k\to0$. For large $k$, $f_k\in C_c^\infty(W)$ and their supports lie in one compact subset of $W$. Since $f_k\,dx\rightharpoonup^*\mu$ and $\|f_k\|_{L^1}\leq\|\mu\|_{\mathcal M}$, Lemma~\ref{lem:approximation} applies. This is convergence of measures in the weak-star sense; a mollified Dirac measure, for example, does not converge to the Dirac measure in total variation.
	
	\subsection{The Green and Poisson operators}\label{sec:duality}
	Fix a bounded open set $W\Subset\Omega_e$. For $f\in L^2(W)$, extended by zero outside $W$, set $F_f(x)=C_{n,s}\int_W f(z)|x-z|^{-n-2s}\,dz$ and $\mathcal P_{q,W}f=(\mathcal G_qF_f)\big|_\Omega$. The Poisson operator $\mathcal P_{q,W}:L^2(W)\to L^2(\Omega)$ is bounded. Its zero-extended interior part is in $X$, but the full solution $\mathcal G_qF_f+f$ need not be in $H^s(\R^n)$ for general $L^2$ data.
	
	For $h\in L^2(\Omega)$, the function $z=\mathcal G_qh$ vanishes in $\Omega_e$, then its exterior fractional Laplacian is
	\begin{equation}\label{eq:green-exterior}
		(-\Delta)^s(\mathcal G_qh)(x)
		=-C_{n,s}\int_\Omega\frac{(\mathcal G_qh)(y)}{|x-y|^{n+2s}}\,dy,
		\quad x\in\Omega_e.
	\end{equation}
	This function is smooth in $\Omega_e$, and on every compact subset each derivative is bounded by $C\|h\|_{L^2(\Omega)}$. Since $\mathcal G_qh$ has zero exterior value, \eqref{eq:green-exterior} is also its nonlocal normal derivative in the sense of \cite[equation~(2.4)]{AKW19}.
	
	For $q=0$ and $L^2$ exterior data, the adjoint relation and a very weak formulation were established in \cite[Lemma~4.2 and Definition~3.3]{AKW19}. Since $W$ is separated from $\overline\Omega$, the function in \eqref{eq:green-exterior} can also be integrated against a finite measure supported in $W$. We give the argument for bounded $q$ and for the solutions in Definition~\ref{def:measure-solution}.
	
	\begin{proposition}[Duality]\label{prop:duality}
		Assume \eqref{eq:nonresonance}. The $L^2$ adjoint of $\mathcal P_{q,W}$ satisfies
		\begin{equation}\label{eq:poisson-adjoint}
			\mathcal P_{q,W}^*h=-(-\Delta)^s(\mathcal G_qh)|_W.
		\end{equation}
		If $\mu$ is a finite signed Borel measure supported in a compact subset of $W$, then its interior solution is characterized by
		\begin{equation}\label{eq:transposition}
			\int_\Omega v_{q,\mu}h\,dx
			=-\int_W(-\Delta)^s(\mathcal G_qh)\,d\mu
			\quad\text{for every }h\in L^2(\Omega).
		\end{equation}
		More precisely, a function $v\in L^2(\Omega)$ satisfies \eqref{eq:transposition} if and only if it is the restriction of $\mathcal G_qF_\mu$ to $\Omega$.
	\end{proposition}
	\begin{proof}
		For $h\in L^2(\Omega)$, the symmetry \eqref{eq:sym_of_Green_kernel} of $\mathcal G_q$ and Fubini's theorem yield
		\[
		\int_\Omega(\mathcal G_qF_\mu)h\,dx=\int_\Omega F_\mu\mathcal G_qh\,dx=-\int_W(-\Delta)^s(\mathcal G_qh)\,d\mu.
		\]
		The last integral is well-defined by the bounds following \eqref{eq:green-exterior}, and it defines a bounded functional on $L^2(\Omega)$. Thus, \eqref{eq:transposition} has the unique solution $\mathcal G_qF_\mu\big|_\Omega$. Its zero extension belongs to $X$ and satisfies \eqref{eq:source}. Taking $\mu=f\,dx$ with $f\in C_c^\infty(W)$ gives \eqref{eq:poisson-adjoint} for smooth $f$. Since $C_c^\infty(W)$ is dense in $L^2(W)$ and $\mathcal P_{q,W}$ is bounded, the identity holds for every $f\in L^2(W)$.
	\end{proof}
	
	For $z\in W$, \eqref{eq:transposition} gives $\int_\Omega v_{q,\delta_z}h\,dx=-(-\Delta)^s(\mathcal G_qh)(z)$. This point evaluation is bounded in $h\in L^2(\Omega)$ since $z$ is separated from $\overline\Omega$.
	
	\subsection{The measured data and reciprocity}\label{sec:reciprocity}
	We now justify \eqref{eq:measurement} and prove reciprocity. Let $u_{q,\mu}=v_{q,\mu}+\mu$ with $K=\supp\mu\Subset\Omega_e$. On $\R^n\setminus(\overline\Omega\cup K)$, the fractional Laplacian is represented by the smooth function
	\begin{equation}\label{eq:measure-exterior}
		(-\Delta)^su_{q,\mu}(x)
		=-C_{n,s}\int_\Omega\frac{v_{q,\mu}(y)}{|x-y|^{n+2s}}\,dy
		-C_{n,s}\int_K\frac{d\mu(y)}{|x-y|^{n+2s}}.
	\end{equation}
	To obtain this formula, test \eqref{eq:distribution} away from $\overline\Omega\cup K$ and apply Fubini's theorem. The positive distance from the supports justifies differentiating under both integrals on compact subsets.
	
	\begin{lemma}\label{lem:exterior-observation}
		Let $W_1,W_2\Subset\Omega_e$ with $\overline{W_1}\cap \overline{W_2}=\emptyset$. For $f\in\wt H^s(W_1)$, the distribution $\Lambda_qf\big|_{W_2}$ is smooth and all its derivatives are bounded on $W_2$. For every multi-index $\alpha$, there is a constant $C_\alpha$ such that $\|\partial^\alpha\Lambda_qf\|_{L^\infty(W_2)}\leq C_\alpha\|f\|_{H^s(\R^n)}$. Consequently, \eqref{eq:measurement} defines a bounded linear functional of $f$ for every finite measure supported in $W_2$.
	\end{lemma}
	\begin{proof}
		The energy solution $u_{q,f}$ is supported in $\overline\Omega\cup\overline{W_1}$. For $x\in W_2$ we have
		\begin{equation}\label{eq:DN-kernel}
			(\Lambda_qf)(x)=-C_{n,s}\int_{\Omega\cup W_1}\frac{u_{q,f}(y)}{|x-y|^{n+2s}}\,dy.
		\end{equation}
		Since $\overline{W_2}$ is separated from $\overline\Omega\cup\overline{W_1}$, we may differentiate under the integral. Cauchy--Schwarz on the bounded support and the estimate for the energy solution give the asserted bounds. Integration against $\mu$ is therefore well-defined.
	\end{proof}
	
	\begin{proposition}[Reciprocity]\label{prop:reciprocity}
		Let $f\in C_c^\infty(W_1)$ and let $\mu$ be a finite signed Borel measure with $\supp\mu\Subset W_2$, where $W_1,W_2$ are as in Lemma~\ref{lem:exterior-observation}. Then
		\begin{equation}\label{eq:reciprocity}
			\int_{W_2}\Lambda_qf\,d\mu=\int_{W_1}f(-\Delta)^su_{q,\mu}\,dx.
		\end{equation}
	\end{proposition}
	\begin{proof}
		Write $u_{q,f}=v_f+f$ and $u_{q,\mu}=v_\mu+\mu$, where $v_f=\mathcal G_qF_f$ and $v_\mu=\mathcal G_qF_\mu$. Since $F_f,F_\mu\in L^2(\Omega)$, the symmetry \eqref{eq:sym_of_Green_kernel} gives
		\[
		\int_\Omega F_\mu v_f\,dx
		=\int_\Omega F_\mu(\mathcal G_qF_f)\,dx
		=\int_\Omega F_f(\mathcal G_qF_\mu)\,dx
		=\int_\Omega F_fv_\mu\,dx.
		\]
		By \eqref{eq:green-exterior} and Fubini's theorem,
		\[
		\int_{W_2}(-\Delta)^sv_f(z)\,d\mu(z)
		=-C_{n,s}\int_\Omega v_f(x)
		\int_{W_2}\frac{d\mu(z)}{|z-x|^{n+2s}}\,dx
		=-\int_\Omega F_\mu v_f\,dx.
		\]
		The same calculation with $v_\mu$ and $f$ gives $\int_{W_1}f(-\Delta)^sv_\mu\,dx
		=-\int_\Omega F_fv_\mu\,dx$. Moreover, since $\supp f\subset W_1$ and $\supp\mu\subset W_2$ are separated,
		\[
		\int_{W_2}(-\Delta)^sf\,d\mu
		=-C_{n,s}\int_{W_2}\int_{W_1}
		\frac{f(x)}{|z-x|^{n+2s}}\,dx\,d\mu(z)
		=\int_{W_1}f(-\Delta)^s\mu\,dx.
		\]
		Combining these identities, we obtain
		\begin{align*}
			\int_{W_2}\Lambda_qf\,d\mu&=-\int_\Omega F_\mu v_f\,dx+\int_{W_2}(-\Delta)^sf\,d\mu\\
			&=-\int_\Omega F_fv_\mu\,dx+\int_{W_1}f(-\Delta)^s\mu\,dx=\int_{W_1}f(-\Delta)^s(v_\mu+\mu)\,dx.
		\end{align*}
		The separation of the supports, the finite total variation of $\mu$, and $v_f,v_\mu\in L^1(\Omega)$ justify the interchanges of integrals. The last integral is an ordinary integral by \eqref{eq:measure-exterior}.
	\end{proof}
	
	For $\mu=\delta_{x_0}$, \eqref{eq:reciprocity} becomes $(\Lambda_qf)(x_0)=\int_{W_1}f(-\Delta)^su_{q,\delta_{x_0}}\,dx$. Thus, all point observations at $x_0$ determine $(-\Delta)^su_{q,\delta_{x_0}}$ on $W_1$.
	
	\subsection{Positive exterior data and comparison of potentials}\label{sec:positive-data}
	For $f\in C_c^\infty(\Omega_e)$, write $v_{q,f}=v_{q,f\,dx}=u_{q,f}-f$. If $q\geq0$, the form $B_q$ is coercive on $X$, so \eqref{eq:nonresonance} holds. We prove positivity under the more general assumption that $B_q$ is coercive.
	
	\begin{lemma}\label{lem:positive-measure}
		Let $0<s<1$ and let $q\in L^\infty(\Omega)$ be such that $B_q$ is coercive on $X$. If $\mu$ is a nonzero finite nonnegative Borel measure with $\supp\mu\Subset\Omega_e$, then for each compact set $D\Subset\Omega$, there is a constant $c_D>0$ such that $v_{q,\mu}\geq c_D$ almost everywhere on $D$. In particular, $v_{q,\mu}>0$ almost everywhere in $\Omega$.
	\end{lemma}

	\begin{proof}
		First, suppose that $w\in X$ satisfies $B_q(w,\phi)\geq0$ for every nonnegative $\phi\in X$. Its negative part $w^-=\max\{-w,0\}$ belongs to $X$, and $(a-b)(a^--b^-)\leq-|a^--b^-|^2$ gives $B_q(w,w^-)\leq-B_q(w^-,w^-)$. By coercivity, we have $w^-=0$, so $w\geq0$ almost everywhere.
		
		Let $K=\supp\mu$. Since $K$ and $\overline\Omega$ are disjoint compact sets and $\mu\geq0$ is nonzero, \eqref{eq:forcing} gives $F_\mu\geq c>0$ on $\Omega$. The comparison argument applied to \eqref{eq:source} shows $v_{q,\mu}\geq0$. Given $D\Subset\Omega$, choose $0\leq\psi\in C_c^\infty(\Omega)$ with $\psi\geq1$ on $D$. Since $(-\Delta)^s\psi$ is bounded and $q\in L^\infty(\Omega)$, we may choose $t>0$ so that $t((-\Delta)^s\psi+q\psi)\leq c$ almost everywhere in $\Omega$. Then $B_q(v_{q,\mu}-t\psi,\phi)\geq0$ for all nonnegative $\phi\in X$. It follows that $v_{q,\mu}\geq t\psi\geq t$ almost everywhere on $D$. A countable compact exhaustion proves the last assertion.
	\end{proof}
	
	\begin{proposition}\label{prop:comparison-identity}
		Let $W_1,W_2\Subset\Omega_e$ with $\overline{W_1}\cap \overline{W_2}=\emptyset$. Let $f\in C_c^\infty(W_1)$, and $\mu$ be a finite signed Borel measure with $\supp\mu\Subset W_2$. If $q_1,q_2\in L^\infty(\Omega)$ satisfy \eqref{eq:nonresonance}, then
		\begin{equation}\label{eq:comparison-identity}
			\mathcal M_{q_1,\mu}(f)-\mathcal M_{q_2,\mu}(f)
			=\int_\Omega(q_1-q_2)v_{q_2,f}v_{q_1,\mu}\,dx.
		\end{equation}
	\end{proposition}
	
	\begin{proof}
		Set $w=v_{q_1,f}-v_{q_2,f}\in X$. Subtracting the equations for $v_{q_1,f}$ and $v_{q_2,f}$ gives $B_{q_1}(w,\phi)=-\int_\Omega(q_1-q_2)v_{q_2,f}\phi\,dx$ for every $\phi\in X$. Since the full solutions have the same exterior value, their difference is $w$. By \eqref{eq:DN-kernel} and Fubini's theorem,
		\[
		\mathcal M_{q_1,\mu}(f)-\mathcal M_{q_2,\mu}(f)=-C_{n,s}\int_\Omega w(y)\int_{W_2}\frac{d\mu(x)}{|x-y|^{n+2s}}\,dy=-\int_\Omega F_\mu w\,dy.
		\]
		The interchange is justified by the separation of $W_2$ from $\overline\Omega$, the finite total variation of $\mu$, and $w\in L^1(\Omega)$. Using \eqref{eq:source} and the equation for $w$ with $\phi=v_{q_1,\mu}$, we obtain $-\int_\Omega F_\mu w\,dy=-B_{q_1}(v_{q_1,\mu},w)=\int_\Omega(q_1-q_2)v_{q_2,f}v_{q_1,\mu}\,dy$, as claimed.
	\end{proof}
	
	\section{Unique continuation for exterior measure data}\label{sec:ucp}
	
	We recall unique continuation on open sets for the fractional Laplacian and then prove the local extension estimates needed to apply \cite[Section~5]{GRSU20} to the solutions of Definition~\ref{def:measure-solution}.
	
	\begin{lemma}[\cite{GSU20}, Theorem~1.2]\label{lem:open-ucp}
		Let $0<s<1$ and let $h\in H^t(\R^n)$ for some $t\in\R$. If $h=(-\Delta)^sh=0$ in a nonempty open subset of $\R^n$, then $h=0$ in $\R^n$.
	\end{lemma}
	
	A finite measure has a bounded Fourier transform and belongs to $H^{-m}(\R^n)$ for every $m>n/2$. For an interior solution $v\in X$ associated with $\mu$, the distribution $u=v+\mu$ belongs to the same space. If $v$ vanishes on a nonempty open subset of $\Omega$, then both $u$ and $(-\Delta)^su$ vanish there. Lemma~\ref{lem:open-ucp} gives $u=0$ on $\R^n$, and restricting to $\Omega_e$ gives $\mu=0$. Thus, a solution associated with a nonzero measure cannot vanish on an interior open set.
	
	\subsection{Local properties of the extension}
	
	Let $a=1-2s\in(-1,1)$. The Poisson kernel with respect to the Caffarelli--Silvestre extension is
	\begin{equation}\label{eq:extension-kernel}
		P_s(x,y)=p_{n,s}\frac{y^{2s}}{(|x|^2+y^2)^{(n+2s)/2}},\quad y>0,
	\end{equation}
	where $p_{n,s}=\Gamma(n/2+s)/(\pi^{n/2}\Gamma(s))$, so $\int_{\R^n}P_s(x,y)\,dx=1$. For $v\in H^s(\R^n)$, its extension $V=P_s*v$ satisfies $\operatorname{div}(y^a\nabla V)=0$ and
	\begin{equation}\label{eq:extension-conormal}
		-\lim_{y\to0}y^a\p_yV=\kappa_s(-\Delta)^sv,
	\end{equation}
	where $\kappa_s=\Gamma(1-s)/(2^{2s-1}\Gamma(s))$ and the limit is in $H^{-s}(\R^n)$; see \cite{CS07}. The constants satisfy $2sp_{n,s}=\kappa_s C_{n,s}$. The kernel formula \eqref{eq:extension-kernel} may also be obtained from the Euclidean heat kernel in \cite[Lemma~6.2, equation~(32)]{CGRU2023reduction}.
	
	For $u=v+\mu$ as in Definition~\ref{def:measure-solution}, let $K=\supp\mu$ and define
	\begin{equation}\label{eq:measure-extension}
		U(x,y)=\int_\Omega P_s(x-z,y)v(z)\,dz
		+\int_K P_s(x-z,y)\,d\mu(z),\quad y>0.
	\end{equation}
	Both integrals converge absolutely since $v\in L^1(\Omega)$ and $\mu$ is finite. Write $B_r'(x_0)=\{x:|x-x_0|<r\}$ and $B_r^+(x_0)=\{(x,y):y>0,\ |x-x_0|^2+y^2<r^2\}$. We omit $x_0$ when it is zero. We use $\|V\|_{H^1(D;y^a)}^2=\int_D y^a(|V|^2+|\nabla V|^2)\,dx\,dy$, and the local trace on $y=0$ belongs to $H^s$ by \cite[Lemma~5.3]{GRSU20}.
	
	\begin{lemma}\label{lem:extension}
		Let $v\in X$ be an interior solution associated with $\mu$, and $B_{2r}'(x_0)\Subset\Omega$. Then $U$ in \eqref{eq:measure-extension} belongs to $H^1(B_{2r}^+(x_0);y^a)$, has trace $v$, and satisfies
		\begin{equation}\label{eq:measure-Robin}
			\int_{B_{2r}^+(x_0)}y^a\nabla U\cdot\nabla\Phi\,dx\,dy
			=-\kappa_s\int_{B_{2r}'(x_0)}qv\Phi(x,0)\,dx
		\end{equation}
		for every smooth test function supported away from the curved boundary.
	\end{lemma}
	
	\begin{proof}
		Write $U=V+M$ with $V=P_s*v$ and $M=P_s*\mu$. The extension theorem gives finite weighted Dirichlet energy for $V$. Also $\|V(\cdot,y)\|_{L^2}\leq\|v\|_{L^2}$, and $a>-1$, so $V$ has finite weighted $L^2$ norm in the half-ball. It has trace $v$ and conormal derivative \eqref{eq:extension-conormal}.
		
		The distance from $\overline{B_{2r}'(x_0)}$ to $K$ is positive. Differentiating under the measure integral gives
		\begin{equation}\label{eq:singular-extension-bounds}
			|M(x,y)|+|\nabla_xM(x,y)|\leq C\|\mu\|_{\mathcal M}y^{2s},\quad |\p_yM(x,y)|\leq C\|\mu\|_{\mathcal M}y^{2s-1}
		\end{equation}
		for $x\in B_{2r}'(x_0)$ and $0<y<2r$. Since $y^a|\p_yM|^2\leq C\|\mu\|_{\mathcal M}^2y^{2s-1}$, the vertical derivative has finite weighted energy near $y=0$. The other terms are integrable as well. Thus, $M\in H^1(B_{2r}^+(x_0);y^a)$ and has zero trace. Moreover,
		\begin{equation}\label{eq:measure-conormal}
			-\lim_{y\to0}y^a\p_yM(x,y)
			=-2sp_{n,s}\int_K\frac{d\mu(z)}{|x-z|^{n+2s}}
			=-\kappa_s F_\mu(x),
		\end{equation}
		locally uniformly in $\Omega$.
		
		Both $V$ and $M$ are weighted harmonic for $y>0$. Since $(-\Delta)^sv+qv=F_\mu$ in $\Omega$, \eqref{eq:extension-conormal} and \eqref{eq:measure-conormal} give $-\lim_{y\to0}y^a\p_yU=-\kappa_s qv$ locally in $H^{-s}$. Integration by parts proves \eqref{eq:measure-Robin}. Since $qv\in L^2(\Omega)$, the boundary pairing is the integral in that formula.
	\end{proof}
	
	\subsection{Unique continuation from a set of positive measure}
	
	We now prove a local version of \cite[Proposition~5.1]{GRSU20}. This is needed here because $v+\mu$ need not belong to $H^s(\R^n)$, although its extension has finite weighted energy locally above $\Omega$.
	
	\begin{proposition}\label{prop:local-ucp}
		Let $1/4\le s<1$, $a=1-2s$, and $b\in L^\infty(B_{2R}')$. Suppose that $U\in H^1(B_{2R}^+;y^a)$ satisfies
		\begin{equation}\label{eq:local-Robin}
			\int_{B_{2R}^+}y^a\nabla U\cdot\nabla\Phi\,dx\,dy
			=\int_{B_{2R}'}b\Tr U\Phi(x,0)\,dx
		\end{equation}
		for every smooth test function supported away from the curved boundary. If $\Tr U=0$ almost everywhere on a measurable set $E\subset B_{2R}'$ with $|E|>0$, then $U=0$ in $B_{2R}^+$.
	\end{proposition}
	
	\begin{proof}
		Suppose $U\not\equiv0$ and choose a Lebesgue density point $x_0\in E$. For small $r>0$, write $H(r)=\|U\|_{L^2(B_r^+(x_0);y^a)}$ and $D(r)=\|\nabla U\|_{L^2(B_r^+(x_0);y^a)}$. We have $H(r)>0$: otherwise interior unique continuation, applied where $y>0$ to an equation with smooth coefficients, would give $U=0$ throughout $B_{2R}^+$.\\
		
		\noindent\emph{Local doubling.}
		We claim that there are constants $r_*>0$ and $C_U$ such that
		\begin{equation}\label{eq:local-H1-doubling}
			H(2r)+rD(2r)\le C_U\big(H(r)+rD(r)\big),\quad 0<r<r_*.
		\end{equation}
		To apply the argument of \cite[Lemma~5.4]{GRSU20} locally, choose $\rho>0$ with $\overline{B_{5\rho}^+(x_0)}\subset B_{2R}^+\cup B_{2R}'$ and set $\widetilde U(x,y)=U(x_0+\rho x,\rho y)$ on $B_5^+$. The weak boundary coefficient in \eqref{eq:local-Robin} becomes $\rho^{2s}b(x_0+\rho x)$; equivalently, the coefficient in the normal-derivative convention of \cite{GRSU20} is $-\rho^{2s}b(x_0+\rho x)$. A radial cutoff, zero near the center and outside $B_4^+$, gives the hypotheses of \cite[Proposition~A.1]{GRSU20}. It contributes no conormal term on the flat boundary because $y^a\p_y\eta\to0$ there. The remaining volume terms lie in the two half-annuli where the cutoff varies and are bounded by the local weighted $H^1$ norm.
		
		The Carleman estimate in \cite[Proposition~A.1]{GRSU20} compares these annular norms as in \cite[Lemma~5.4]{GRSU20}. For $s>1/4$, its boundary term is absorbed by taking the Carleman parameter sufficiently large. At $s=1/4$, first decrease $\rho$ so that $\rho^{2s}\|b\|_{L^\infty(B_{2R}')}$ is small enough for the boundary absorption in Step~2 of \cite[Lemma~5.4]{GRSU20}. The norm on the fixed outer half-annulus is positive by interior unique continuation and gives a finite constant depending on $U$. This proves \eqref{eq:local-H1-doubling} from \eqref{eq:local-Robin} and the local weighted-energy assumption.
		
		By \cite[Lemmas~5.2 and~5.3]{GRSU20}, Caccioppoli's inequality and the trace estimate give
		\begin{align*}
			D(r/4)&\leq Cr^{-1}H(r/2)+C\|b\|_\infty^{1/2}\|\Tr U\|_{L^2(B_{r/2}'(x_0))},\\
			\|\Tr U\|_{L^2(B_{r/2}'(x_0))}&\leq C\big(r^sD(r)+r^{s-1}H(r)\big).
		\end{align*}
		Two applications of \eqref{eq:local-H1-doubling} give $D(r)\leq C_U(D(r/4)+r^{-1}H(r/4))$. Combining these bounds and making $r_*$ smaller to absorb $C_U\|b\|_\infty^{1/2}r^sD(r)$, we obtain $D(r)\leq C_Ur^{-1}H(r)$, as in \cite[Lemma~5.7]{GRSU20}. Thus \eqref{eq:local-H1-doubling} becomes
		\begin{equation}\label{eq:local-bulk-doubling}
			H(2r)+rD(2r)\leq C_UH(r),\quad 0<r<r_*.
		\end{equation}
		
		\smallskip
		\noindent\emph{Rescaling and compactness.}
		Put $d_a=n+1+a=n+2-2s$ and define $U_r(x,y)=r^{d_a/2}U(x_0+rx,ry)/H(r)$ on $B_2^+$. By \eqref{eq:local-bulk-doubling},
		\begin{equation}\label{eq:normalized-extension}
			\|U_r\|_{L^2(B_1^+;y^a)}=1,
			\quad \|U_r\|_{H^1(B_2^+;y^a)}\le C_U.
		\end{equation}
		The trace theorem bounds $\Tr U_r$ in $H^s(B_{7/4}')$, which is compactly embedded in $L^2(B_{3/2}')$ for every $n\geq1$ and $s>0$.
		
		We also need strong convergence in the weighted $L^2$ space. For $V\in H^1(B_2^+;y^a)$ and $S_\eps=B_{3/2}'\times(0,\eps)$, with $\eps>0$ small, weighted Cauchy--Schwarz on vertical segments gives
		\begin{equation}\label{eq:thin-strip}
			\int_{S_\eps}y^a|V|^2\,dx\,dy
			\leq C_s\eps^{a+1}\|\Tr V\|_{L^2(B_{3/2}')}^2
			+C_s\eps^2\int_{S_\eps}y^a|\p_yV|^2\,dx\,dy.
		\end{equation}
		For smooth $V$, use $|V(x,y)|^2\leq2|V(x,0)|^2+2y^{1-a}(1-a)^{-1}\int_0^y t^a|\p_yV(x,t)|^2\,dt$ and integrate in $y$; density gives \eqref{eq:thin-strip} for general $V$. Since $a>-1$, its right-hand side tends to zero uniformly on bounded subsets of $H^1(B_2^+;y^a)$. Away from $y=0$ the weight is bounded above and below, and ordinary Rellich compactness applies. A diagonal subsequence and \eqref{eq:thin-strip} give compactness in $L^2(B_{3/2}^+;y^a)$.
		
		We may choose $r_k\to 0$ and a subsequence such that
		\begin{equation}\label{eq:extension-limits}
			U_{r_k}\rightharpoonup U_0\text{ in }H^1(B_{3/2}^+;y^a),\quad 
			U_{r_k}\to U_0\text{ in }L^2(B_{3/2}^+;y^a),\quad 	\Tr U_{r_k}\to\Tr U_0\text{ in }L^2(B_{3/2}').
		\end{equation}
		In particular, $\|U_0\|_{L^2(B_1^+;y^a)}=1$.\\

		\noindent\emph{Boundary values of the limit.}
		Let $E_k=\{x\in B_1':x_0+r_kx\in E\}$. Since $x_0$ is a density point, $|B_1'\setminus E_k|\to0$, while $\Tr U_{r_k}=0$ almost everywhere on $E_k$. By strong convergence of the traces,
		\[
		\int_{B_1'}|\Tr U_0|^2\,dx\leq\|\Tr U_0-\Tr U_{r_k}\|_{L^2(B_1')}^2+\int_{B_1'\setminus E_k}|\Tr U_0|^2\,dx\longrightarrow0.
		\]
		The second term tends to zero by absolute continuity. Thus $\Tr U_0=0$ on $B_1'$.
		
		Rescaling \eqref{eq:local-Robin} gives
		\[
		\int_{B_{3/2}^+}y^a\nabla U_{r_k}\cdot\nabla\Phi\,dx\,dy=\int_{B_{3/2}'}r_k^{2s}b(x_0+r_kx)\Tr U_{r_k}\Phi(x,0)\,dx
		\]
		for tests supported away from the curved boundary. The traces are bounded in $L^2$, and $r_k^{2s}\to0$. Passing to the limit, we find that $U_0$ is weighted harmonic and has zero conormal derivative on $B_{3/2}'$.
		
		Even reflection across $y=0$ gives a weak solution of $\operatorname{div}(|y|^a\nabla U_0)=0$ in $B_{3/2}$. Since $|y|^a$ is an $A_2$ weight, the local boundedness theorem of \cite{fabes1982local} applies. On $B_1'$, both the trace and conormal derivative of $U_0$ vanish, so \cite[Proposition~2.2]{ruland2015unique} gives $U_0=0$ on a smaller half-ball. Interior unique continuation then gives $U_0=0$ throughout $B_{3/2}^+$. This contradicts the normalization in \eqref{eq:normalized-extension} and the strong convergence in \eqref{eq:extension-limits}.
	\end{proof}
	
	\begin{corollary}\label{cor:measure-ucp}
		Let $1/4\le s<1$ and $q\in L^\infty(\Omega)$. Suppose that $v\in X$ is an interior solution associated with $\mu$, and set $u=v+\mu$. If $v=0$ almost everywhere on a measurable set $E\subset\Omega$ with $|E|>0$, then $u=0$ in $\mathscr D'(\R^n)$, and consequently $v=0$ and $\mu=0$.
	\end{corollary}
	
	\begin{proof}
		Choose $B_{2r}'(x_0)\Subset\Omega$ with $|E\cap B_{2r}'(x_0)|>0$. Lemma~\ref{lem:extension} and Proposition~\ref{prop:local-ucp}, with $b=-\kappa_s q$, show that the extension of $u$ vanishes in $B_{2r}^+(x_0)$. Its trace gives $u=v=0$ in $B_{2r}'(x_0)$, and the equation gives $(-\Delta)^su=0$ there. Since $u\in H^{-m}(\R^n)$ for $m>n/2$, Lemma~\ref{lem:open-ucp} implies $u=0$ on $\R^n$. Restricting to $\Omega_e$ gives $\mu=0$, and then $v=0$.
	\end{proof}
	
	\section{Proofs of the main results}\label{sec:proof}
	\begin{proof}[Proof of Theorem~\ref{thm:uniqueness}]
		For $j=1,2$, write $v_j=v_{q_j,\mu}$ and $u_j=v_j+\mu$. By \eqref{eq:equal-data} and Proposition~\ref{prop:reciprocity}, $\int_{W_1}f(-\Delta)^s(u_1-u_2)\,dx=0$ for every $f\in C_c^\infty(W_1)$. Since the same measure occurs in both solutions, $u_1-u_2=v_1-v_2\in X$. This difference and its fractional Laplacian vanish on $W_1$, so Lemma~\ref{lem:open-ucp} gives $v_1=v_2$ in $\R^n$. Set $v=v_1=v_2$. Subtracting \eqref{eq:source} for the two potentials yields
		\begin{equation}\label{eq:potential-product}
			(q_1-q_2)v=0\quad\text{almost everywhere in }\Omega.
		\end{equation}
		
		First, suppose $q_1,q_2\in C^0(\overline\Omega)$. If $q_1-q_2$ is nonzero at an interior point, there is a ball $B\Subset\Omega$ on which $|q_1-q_2|$ is bounded below by a positive constant. By \eqref{eq:potential-product}, $v=0$ almost everywhere in $B$. Since $\mu$ is supported in $\Omega_e$, $u_2=0$ in $B$, and its equation gives $(-\Delta)^su_2=0$ there. As $u_2=v+\mu\in H^{-m}(\R^n)$ for $m>n/2$, Lemma~\ref{lem:open-ucp} gives $u_2=0$ on $\R^n$, contradicting $\mu\neq0$. Thus $q_1=q_2$ in $\Omega$ and, by continuity, on $\overline\Omega$.
		
		Next, suppose $q_1,q_2\in L^\infty(\Omega)$ and $s\geq1/4$. If they differ on a set of positive measure, then $E_j=\{x\in\Omega:|q_1(x)-q_2(x)|>1/j\}$ has positive measure for some $j\in\N$. Equation \eqref{eq:potential-product} gives $v=0$ almost everywhere on $E_j$. Applying Corollary~\ref{cor:measure-ucp} to $u_2=v+\mu$ gives $\mu=0$, a contradiction. Hence $q_1=q_2$ almost everywhere in $\Omega$.
	\end{proof}
	
	\begin{proof}[Proof of Theorem~\ref{thm:single-input-point}]
		Choose $W_2\Subset\Omega_e$ containing $x_0$ with $\overline W\cap\overline{W_2}=\emptyset$. Proposition~\ref{prop:comparison-identity}, with $W_1=W$ and $\mu=\delta_{x_0}$, gives $\int_\Omega(q_1-q_2)v_{q_2,f}v_{q_1,\delta_{x_0}}\,dx=0$. Lemma~\ref{lem:positive-measure} applies to $f\,dx$ and $\delta_{x_0}$, so both interior solutions are positive almost everywhere. Since $q_1-q_2\geq0$, the integral can vanish only if $q_1=q_2$ almost everywhere in $\Omega$.
	\end{proof}
	
	The same argument applies when $\mu$ is any nonzero finite nonnegative measure supported compactly in $W_2$ and $0\not\equiv f\in C_c^\infty(W_1)$ is nonnegative. If $0\leq q_2\leq q_1$, then $\mathcal M_{q_1,\mu}(f)=\mathcal M_{q_2,\mu}(f)$ implies $q_1=q_2$ almost everywhere in $\Omega$.

	\begin{proof}[Proof of Corollary~\ref{cor:finitely_many_points}]
		Let $Q=\{q\in L^\infty(\Omega;\R):B_q\text{ is coercive on }X\}$. This is an open set, and it contains every nonnegative potential by the fractional Poincar\'e inequality. Lemma~\ref{lem:positive-measure} gives $v_{q,f}>0$ almost everywhere in $\Omega$ for $q\in Q$.
		
		Suppose $q_1,q_2\in Q$ and $(\Lambda_{q_1}-\Lambda_{q_2})f=0$ on $W_2$. The difference $v_{q_1,f}-v_{q_2,f}$ and its fractional Laplacian vanish on $W_2$. By Lemma~\ref{lem:open-ucp}, $v_{q_1,f}=v_{q_2,f}$. Subtracting their equations gives $(q_1-q_2)v_{q_2,f}=0$ almost everywhere in $\Omega$. Positivity of $v_{q_2,f}$ implies $q_1=q_2$.
		
		Identify $V$ with $\R^d$ and put $A=Q\cap V$. Since $Q$ is open in $L^\infty(\Omega)$, $A$ is open in $V$ and is a $d$-dimensional analytic model class. On $A$, the operator induced by $B_q$ depends affinely on $q$ and is invertible, so its inverse $\mathcal G_q$ depends real analytically on $q$. For $x\in W_2$, \eqref{eq:DN-kernel} and $u_{q,f}=\mathcal G_qF_f+f$ show that $\mathfrak m(q,x):=(\Lambda_qf)(x)$ is real analytic jointly in $(q,x)\in A\times W_2$: the kernel $|x-y|^{-n-2s}$ is real analytic in $x$, uniformly for $y\in\overline\Omega\cup\supp f$ on compact subsets of $W_2$. By the preceding paragraph, $\mathfrak m(q_1,\cdot)-\mathfrak m(q_2,\cdot)$ is not identically zero on $W_2$ whenever $q_1\ne q_2$ in $A$. Thus \cite[Theorem~2.6]{alberti2026analytic}, applied with $Z=W_2$ and normalized Lebesgue measure, shows that for Lebesgue-almost every choice of $2d+1$ points in $W_2$, the corresponding observations distinguish all $q\in A$. Since every nonnegative potential in $V$ lies in $A$, this proves the corollary.
	\end{proof}

	\section*{Statements and declarations}
	
	\para{Data availability}
	No datasets were generated or analyzed during the current study.
	
	\para{Conflict of interests}
	The authors declare that they have no conflicts of interest.
	
	\para{Use of AI tools}
	An initial manuscript containing the mathematical results and arguments was prepared by the authors before ChatGPT Pro was used. ChatGPT Pro was used only to assist with language editing and organization of the text. The authors reviewed the revisions and take full responsibility for the final manuscript.
	
	\para{Acknowledgments} The question studied in this paper took shape through discussions with Tobias Weth during Y.-H. L.'s visit to Goethe University Frankfurt. We are especially grateful to him for his insights in the early development of this work and for his generosity and kindness. Y.-H. L. is partially supported by the National Science and Technology Council (NSTC) of Taiwan under project 113-2628-M-A49-003. Y.-H. L. acknowledges financial support from the Alexander von Humboldt Foundation through the Henriette Herz Scouting Programme and the hospitality of Universit\"at Duisburg-Essen.
	
	\bibliographystyle{alpha}
	\bibliography{refs}
\end{document}